\documentclass[12pt,leqno,twoside]{amsart}

\usepackage{amssymb,amsmath,amsthm,soul,color}

\usepackage[noadjust]{cite}

\usepackage{t1enc}

\usepackage[utf8]{inputenc}
\usepackage{indentfirst}

\usepackage{todonotes}

\usepackage[left=2.5cm, right=2.5cm, top=2cm, bottom=2cm]{geometry}

\usepackage{url}
\usepackage{soul}

\usepackage{mathrsfs}

\usepackage{enumitem,graphicx}
\usepackage[colorlinks=true,urlcolor=blue,
citecolor=red,linkcolor=blue,linktocpage,pdfpagelabels,
bookmarksnumbered,bookmarksopen]{hyperref}

\usepackage[hyperpageref]{backref}
\usepackage[metapost]{mfpic}
\usepackage[normalem]{ulem}

\newtheorem{Th}{Theorem}[section]
\newtheorem{Prop}[Th]{Proposition}
\newtheorem{Lem}[Th]{Lemma}

\newtheorem{Rem}[Th]{Remark}

\newenvironment{altproof}[1]
{\noindent
{\em Proof of {#1}}.}
{\nopagebreak\mbox{}\hfill $\Box$\par\addvspace{0.5cm}}

\newcommand{\wt}{\widetilde}

\newcommand{\vp}{\varphi}

\newcommand{\eps}{\varepsilon}

\newcommand{\R}{\mathbb{R}}

\newcommand{\Z}{\mathbb{Z}}

\newcommand{\cA}{{\mathcal A}}
\newcommand{\cB}{{\mathcal B}}
\newcommand{\cC}{{\mathcal C}}

\newcommand{\cI}{{\mathcal I}}

\newcommand{\cM}{{\mathcal M}}
\newcommand{\cN}{{\mathcal N}}

\newcommand{\cS}{{\mathcal S}}

\newcommand{\Om}{\Omega}

\newcommand{\weakto}{\rightharpoonup}

\numberwithin{equation}{section}

\begin{document}

\title[Multiple normalized solutions]{
	Multiple normalized solutions to a system of nonlinear Schr\"{o}dinger equations}

\author[J. Mederski]{Jaros\l aw Mederski}
\author[A. Szulkin]{Andrzej Szulkin}

\address[J. Mederski]{\newline\indent
	{Institute of Mathematics,
		\newline\indent 
		Polish Academy of Sciences,
		\newline\indent 
		ul. \'Sniadeckich 8, 00-656 Warsaw, Poland,		\newline\indent 
		and,		\newline\indent 
		Faculty of Mathematics and Computer Science,		\newline\indent 
		Nicolaus Copernicus University, 		\newline\indent 
		ul. Chopina 12/18, 87-100 Toruń, Poland 
}
}
\email{\href{mailto:jmederski@impan.pl}{jmederski@impan.pl}}

\address[A. Szulkin]{\newline\indent 
	Department of Mathematics,
	\newline\indent 
	Stockholm University,
	\newline\indent 
	106 91 Stockholm, Sweden
}
\email{\href{mailto:andrzejs@math.su.se}{andrzejs@math.su.se}}


\maketitle

\begin{abstract} 
We find a normalized solution $u=(u_1,\ldots,u_K)$ to the system of $K$ coupled nonlinear Schr\"odinger equations 
\begin{equation*}
	\left\{ \begin{array}{l}
		-\Delta u_i+ \lambda_i u_i = \sum_{j=1}^K\beta_{i,j}u_i|u_i|^{p/2-2}|u_j|^{p/2} \quad \mathrm{in} \ \R^3, \\
		u_i \in H^1_{rad}(\R^3), \\
		\int_{\R^3} |u_i|^2 \, dx = \rho_i^2 \quad \text{for }i=1,\ldots, K,
	\end{array} \right. 
\end{equation*}
where $\rho=(\rho_1,\ldots,\rho_K)\in(0,\infty)^K$ is prescribed, $(\lambda,u) \in\R^K\times H^1(\R^3)^K$ are the unknown and $4\leq p<6$. In the case of  two equations we show the existence of multiple solutions provided that the coupling is sufficiently large. We also show that for negative coupling there are no ground state solutions. The main novelty in our approach is that we use  the Cwikel-Lieb-Rozenblum theorem in order to estimate the Morse index of a solution as well as a Liouville-type result in an exterior domain.
\medskip

\end{abstract}

\section{Introduction}

 Consider the system of time-dependent nonlinear Schr\"odinger equations of the form
\begin{equation*}
\mathrm{i}\frac{\partial \Phi_i}{\partial t}
-\Delta \Phi_i = g_i(\Phi),\quad i=1,\ldots,K.
\end{equation*}
In this paper we look for solutions which are normalized by prescribing the $L^2$-bounds for $\Phi = (\Phi_1,\ldots,\Phi_K)$:
\[
\int_{\R^3}\left|\Phi_i(t,x)\right|^2\,dx=\rho_i^2\quad\hbox{for }i=1,\ldots,K.
\]
Here $\rho=(\rho_1,\ldots,\rho_K)\in(0,\infty)^K$ represents the total number of atoms in Bose--Einstein condensation \cite{Malomed,Frantzeskakis} or
the power supply in nonlinear optics \cite{Buryak}. For a particular power-type nonlinearity in three dimensions,
seeking for solitary wave solutions of the form $\Phi_i(x,t)=u_i(x)e^{-\mathrm{i}\lambda_i t}$ leads to the 
 system of time-independent Schr\"odinger equations
\begin{equation}\label{eq}
\left\{ \begin{array}{l}
-\Delta u_i+ \lambda_i u_i = \sum_{j=1}^K\beta_{i,j}u_i|u_i|^{p/2-2}|u_j|^{p/2} \quad \mathrm{in} \ \R^3, \\
u_i \in H^1_{rad}(\R^3), \\
\int_{\R^3} |u_i|^2 \, dx = \rho_i^2 \quad \text{for }i=1,\ldots, K,
\end{array} \right. ,
\end{equation}
where we assume $\beta_{i,i}>0$ and $\beta_{i,j}=\beta_{j,i}\in\R$ for $i,j=1,\ldots,K$.
Here 
\[
(\lambda,u)=(\lambda_1,\ldots,\lambda_K,u_1,\ldots,u_K) \in\R^K\times H^1_{rad}(\R^3)^K
\]
 are the unknown and $H^1_{rad}(\R^3)^K$ is the subspace of $H^1(\R^3)^K$ whose all components are radial. 
Note that this covers the most important example from the physical point of view which is  $p=4$; see \cite{Buryak,Malomed,Esry,LSSY} and the references therein. We look for solutions in $H^1_{rad}(\R^3)^K$ and not in $H^1(\R^3)^K$ in order to get some additional compactness for the functional $J$ defined below.

Recall that Kwong \cite{Kwong} showed that there exists a unique positive radial solution to the equation
$$-\Delta u + u=u^{p-1}.$$
After the rescaling 
$$u_{\beta_{1,1}}:=\alpha u(\gamma \cdot),$$
for suitable $\alpha,\gamma>0$ this gives a unique positive radial solution to \eqref{eq} with $K=1$.  
 In general, if $K\geq 2$, a scaling-type argument fails.
It is well known that weak solutions to \eqref{eq}  are critical points of the energy functional $J|_\cS$ given by
$$J(u):=\frac12\int_{\R^3}|\nabla u|^2\,dx-\frac1p\sum_{1\leq i, j\leq K}\beta_{i,j} \int_{\R^3}|u_i|^{p/2}|u_j|^{p/2}\,dx,\quad u\in H^1_{rad}(\R^3)^K$$
where
\begin{equation} \label{eq:S}
	\cS := \left\{ u = (u_1,\ldots,u_K) \in H^1_{rad}(\R^3)^K \ : \ \int_{\R^3} |u_i|^2 \, dx = \rho_i^2\quad\text{for } i=1,\ldots,K \right\}.
\end{equation}
The constants $\lambda_i \in \R$ in \eqref{eq} appear as the Lagrange multipliers. Similarly as for one equation, a global minimizer for $J$ restricted to $\cS$ can only exist in the  $L^2$-subcritical case ($2<p<\frac{10}{3}$), and in the $L^2$-critical case ($p =\frac{10}{3}$) provided that $\rho_i$ are sufficiently small  \cite{Stuart,Lions84,Schino,BiegMed}.  If $p>\frac{10}{3}$, then $\inf_\cS J=-\infty$ and the purpose of this work is to construct a minimax theory for $J$ on $\cS$. Therefore, from now on we assume that $\frac{10}{3}<p<6$.

Let us recall that every critical point of the functional above belongs to $W^{2,q}_\textup{loc}(\R^3)^K$ for all $q < \infty$ and satisfies the Poho\v{z}aev \cite{BrezisLiebCMP84} and the Nehari identities
$$
\int_{\R^3} |\nabla u|^2 \, dx = \frac6p\sum_{1\leq i, j\leq K}\beta_{i,j} \int_{\R^3}|u_i|^{p/2}|u_j|^{p/2}\,dx - 3 \sum_{i=1}^K \lambda_i \int_{\R^3}|u_i|^2 \, dx
$$
and
$$
J'(u)[u]+\sum_{i=1}^K\lambda_i \int_{\R^3} |u_i|^2 \, dx = 0.
$$
By a linear combination of the two equalities above it is easily checked that every solution $u$ satisfies the {\em Nehari-Poho\v{z}aev identity}
\begin{equation} \label{eq:M}
M(u):= \int_{\R^3} |\nabla u|^2 \, dx - \frac{3(p-2)}{2p}\sum_{1\leq i, j\leq K}\beta_{i,j} \int_{\R^3}|u_i|^{p/2}|u_j|^{p/2}\,dx = 0.
\end{equation}
In a usual way (cf. e.g. \cite{BartschJS2016,BartschSoaveJFA,MederskiSchino}) we introduce the constraint
\begin{equation} \label{eq:cM}
\cM := \left\{ u \in H^1_{rad}(\R^3)^K \setminus \{0\}:  M(u)=0 \right\}
\end{equation}
which contains all nontrivial solutions to \eqref{eq} and does not depend on $\lambda=(\lambda_1,\ldots,\lambda_K)$.

In the case of two equations Bartsch, Jeanjean and Soave proved in \cite{BartschJS2016} that for $p=4$ and every sufficiently small or sufficiently large $\beta :=\beta_{1,2}= \beta_{2,1}>0$  there exists a  solution  $(u,\lambda)\in\cS\times(0,\infty)^2$ to \eqref{eq}. See also \cite{BartschZhongZou} for an extension and \cite{BartschJeanjean,LiZou} for more general exponents, or \cite{BartschSoaveJFA,BartschSoaveJFACorr} for $\beta<0$.  Moreover, each component of $u$ is positive, radial, radially nonincreasing, of class $\cC^2$ and
$$J(u)=\inf_{\cS\cap\cM}J$$
provided that $\beta>0$ is  sufficiently large -- see \cite{MederskiSchino} where also a more general nonlinearity has been considered.

While there are several multiplicity results concerning normalized solutions for a single equation (under different assumptions on the right-hand side), the problem of multiple normalized solutions to \eqref{eq} with $K\ge 2$ has been little investigated. To our knowledge there are only a few results, even in the case $p=4$ which is most interesting from the point of view of physics. Here we can mention that
a second positive solution to \eqref{eq} has been found by Gou and Jeanjean \cite{GouJeanjean} for $\beta>0$ small and $K=2$. Next, Bartsch and Soave \cite{BartschSoaveCalV} observed that the system \eqref{eq} with $K=2$ is invariant with respect to the involution $(u_1,u_2)\mapsto (u_2,u_1)$ provided that $\rho_1=\rho_2$, $\beta_{1,1}=\beta_{2,2}$ and $p=4$. Inspired by \cite{DancerWeiWeth} they showed that for any $\beta<-\beta_{1,1}=-\beta_{2,2}$ the problem \eqref{eq} has infinitely many pairs $(u_1,u_2)$, $(u_2,u_1)$ ($u_1\ne u_2$) of positive radial solutions  in this particular situation.

The main purpose of this work is to show that if $K=2$ and $m\ge 1$, then for  $\beta_{1,2}= \beta_{2,1}$ large enough there exist at least $m$ distinct solutions to \eqref{eq}. We impose no restrictions on $\rho_i$ and $\beta_{i,i}>0$.

 We call a solution to \eqref{eq} a \emph{ground state} if it minimizes $J$ on $\cS\cap\cM$. 

Now we state the main result of this paper.

\begin{Th}\label{th:mainK2} 
Let $4\leq p<6$ and suppose that $K=2$. For any $m\geq 1$ there is a  constant $\beta_m>0$ such that if 
	$\beta := \beta_{1,2} = \beta_{2,1}>\beta_m$,
	then \eqref{eq} either has infinitely many solutions or at least $m$ solutions $u^j\in H^1_{rad}(\R^3)^2$ such that $J(u^1)<J(u^2)<\cdots<J(u^m)$. One of these solutions has all components positive and is a ground state.
\end{Th}

 Note that the system \eqref{eq} is $\Z_2^K$-invariant where $\Z_2 := \{\pm 1\}$, i.e. if $u = (u_1,\ldots, u_K)$ is a solution to \eqref{eq}, then so are $(\pm u_1,\ldots, \pm u_K)$ (any combination of signs). These solutions will be called the $\Z_2^K$-\emph{orbit} (or simply the \emph{orbit}) of $u$. In particular, there are at least $m$ different $\Z^2_2$-orbits of solutions in Theorem \ref{th:mainK2}.

For general systems with $K\ge 2$  and $\beta_{i,j}>0$ we are only able to prove the existence of one solution under rather restrictive assumptions.

\begin{Th}\label{th:mainK}  
Let $4\leq p\le\frac{14}3$ and $K\geq 2$. If
	\begin{equation}\label{eq:betacond}
\max_{\substack{\cI\subset\{1,\ldots,K\}, \\1\leq |\cI|\leq K-1}}\Big(\max_{i\in\cI}\{\beta_{i,i}\rho_i^{\frac{6-p}2}\}+\frac{|\cI|-1}{|\cI|^{\frac{3p-10}{4}}}\max_{\substack{i,j\in\cI \\i\neq j}}\big\{\beta_{i,j}(\rho_i\rho_j)^{\frac{6-p}4}\big\}\Big)<\frac{\sum_{1\leq i, j\leq K}\beta_{i,j}(\rho_i\rho_j)^{p/2}}{\big(\sum_{i=1}^K\rho_i^2\big)^{\frac34(p-2)}}
\end{equation}
and $\beta_{i,j}>0$ for all $i,j$, then \eqref{eq} has a solution $u\in H^1_{rad}(\R^3)^K$ such that $J(u)=\inf_{\cS\cap \cM}J$. So $u$ is a ground state and we may assume all its components are positive.
\end{Th}

Here $|\cI|$ is the cardinality of the set $\cI$. Note  that if $K=2$, then $|\cI|-1=0$ and \eqref{eq:betacond} holds for $\beta:=\beta_{1,2}=\beta_{2,1}$ sufficiently large.
For general $K\ge 2$ the following is true.

\begin{Prop}
A sufficient condition for \eqref{eq:betacond} to hold is that $4\le p\le \frac{14}3$, $\rho_i=\rho$ for all $i$, $\beta_{i,j} = \beta$ for all $i\ne j$ and $\beta$ is sufficiently large.
\end{Prop}

\begin{proof}
If $\rho_i=\rho$ and $\beta_{i,j}=\beta$, then \eqref{eq:betacond} can be re-written as
\begin{equation} \label{eq:ineqK}
\max_{\substack{\cI\subset\{1,\ldots,K\}, \\1\leq |\cI|\leq K-1}}\Big(\frac1\beta\max_{i\in\cI}\{\beta_{i,i}\}+\frac{|\cI|-1}{|\cI|^{\frac{3p-10}{4}}}\Big) < \frac{\frac1\beta \sum_{i=1}^K\beta_{i,i}  + K(K-1)}{K^{\frac34(p-2)}}.
\end{equation}
So the conclusion will follow if we can show that $(K-2)K^{\frac34(p-2)}(K-1)^{-\frac{3p-10}4} < K(K-1)$, i.e.
\begin{equation} \label{eq:K}
(K-2)K^{\frac{3p-10}4} < (K-1)^{\frac{3p-6}4}. 
\end{equation}
Note that here we have replaced $|\cI|$ in \eqref{eq:ineqK} with $K-1$ because the function $t\mapsto (t-1)/t^{\frac{3p-10}4}$ is increasing for $t>0$. 
By the inequality of arithmetic and geometric means we get
$$
(K-2)^{\frac{4}{3p-6}}K^{\frac{3p-10}{3p-6}}<\frac{4}{3p-6}(K-2)+\frac{3p-10}{3p-6}K\leq K-1
$$
where the second inequality is a consequence of the assumption $p\le\frac{14}3$. Clearly, this implies \eqref{eq:K}.
\end{proof}

\begin{Rem}
\emph{
We would like to mention that a variant of the condition \eqref{eq:betacond} has appeared in \cite{BartschJS2016} for a system of $K$ equations in the particular case $p=4$. However, \cite[Lemma 5.3]{BartschJS2016} is not true as stated. It will become true after replacing
the condition \cite[(1.14)]{BartschJS2016} appearing in  \cite[Theorem 1.5]{BartschJS2016}  with
\begin{equation*}
	\frac{\Big(\sum_{i=1}^ka_i^2\Big)^3}{\Big(\sum_{1\leq i,j\leq k}\beta_{ij}a_i^2a_j^2\Big)^2} < \min_{\substack{\cI\subset\{1,\ldots,k\}\\1\leq |\cI|\leq k-1}}\frac{1}{\Big(\max_{i\in\cI}\{\beta_{ii}a_i\}+\frac{k-2}{\sqrt{k-1}}\max_{i,j\in\cI,i\neq j}\{\beta_{ij}a_i^{1/2}a_j^{1/2}\}\Big)^2}.
	\end{equation*}
Under this condition \cite[Theorem 1.5]{BartschJS2016} is a special case of Theorem \ref{th:mainK} in the present paper.
}
\end{Rem}

As will be seen later, we shall give a variant of condition \eqref{eq:betacond} which implies the existence of at least $m$ solutions $u^j\in H^1_{rad}(\R^3)^K$ for any $K\geq 2$ (see Lemma \ref{lem:c_0_general} below). However, if $m\geq 2$, we are able to check this condition only for $K=2$ and it is an open question whether Theorem \ref{th:mainK2} can be generalized for systems with at least three equations.

Even for the system with $K=2$ one encounters a considerable difficulty in the variational approach because not all Palais-Smale sequences need to converge strongly. In $H^1_{rad}(\R^3)^2$, if $(u^n)$ is a bounded Palais-Smale sequence, then $u^n\rightharpoonup u=(u_1,u_2)$ after passing to a subsequence and one can show $J(u^n)\to J(u)$. However, since the embedding of  $H^1_{rad}(\R^3)^2$ into $L^2(\R^3)^2$ is not compact,  $u_1,u_2$ may not have the $L^2$-norms as required. In many cases the strong convergence can be shown at the ground state level but it is unclear what happens at higher levels. The only paper we know of where this issue could be dealt with is the already mentioned \cite{BartschSoaveCalV}. Unfortunately, the approach used there cannot be carried over to more general situations.

In order to get multiple normalized solutions to \eqref{eq} we study the energy functional $J|_\cS$ and the crucial step is to show that each Palais-Smale sequence at $c$ which approaches $\cM$ either has a convergent subsequence or converges weakly to a semitrivial solution of the equations in \eqref{eq}, see Theorem \ref{th:PS} below. The proof is based on the Morse index, Cwikel-Lieb-Rozenblum  theorem \cite{Cwikel,Lieb,Rozenblum} and a Liouville-type result in an exterior domain (see Lemma \ref{lemLu}).

\begin{Th}\label{th:PS} Let $c>0$ and $4\leq p<6$. 
Let $(u^n)\subset \cS$ be a sequence such that $J(u^n)\to c$, $J|_{\cS}'(u^n)\to 0$ and $M(u^n)\to 0$ as $n\to\infty$. Then $(u^n)$ is bounded, so $u^n\rightharpoonup u = (u_1,\ldots,u_K)$ after passing to a subsequence, $J(u)=c$ (hence $u\ne 0$) and for each $i = 1,\ldots,K$ either $u_i=0$ or $u_i^n\to u_i$ in $H_{rad}^1(\R^3)$. 
\end{Th}

We emphasize that although in this paper we mainly focus on the case $\beta_{i,j} > 0$, the result above holds true regardless of the signs of $\beta_{i,j}$ for $i\ne j$.

In \cite{BartschSoaveJFA} it has been shown that for $K=2$, $p=4$ and $\beta<0$ the system \eqref{eq} has a solution which is of mountain pass type. The same case has also been considered in \cite{IkomaTanaka} where (among other things) a different and simpler argument was given.  A natural question to ask is whether there also exists a ground state. We show in the next theorem that this is not the case. 
Let $u_{\beta_{i,i}}$ be the unique positive solution to
\begin{equation*}
\left\{ \begin{array}{l}
-\Delta u+ \lambda u = \beta_{i,i}u^{p-1} \quad \mathrm{in} \ \R^3, \\
u \in H^1_{rad}(\R^3), \\
\int_{\R^3} |u|^2 \, dx = \rho_i^2.
\end{array} \right.
\end{equation*}
As before, $\lambda$ appears as a Lagrange multiplier.

\begin{Th}\label{th:nonexistence}
Let $\frac{10}3< p<6$ and suppose that $K\ge 2$ and $\beta_{i,j}<0$ for all $i\neq j$. Then 
$$\inf_{\cS\cap\cM}J=\min\big\{J(0,\ldots,u_{\beta_{i,i}},\dots, 0):\; i=1,\ldots,K\big\}$$
and \eqref{eq} has no radial ground state solution, i.e.  $\inf_{\cS\cap\cM}J$ is not attained. The same holds true if $H^1_{rad}(\R^3)$ is replaced by $H^1(\R^3)$ in \eqref{eq} \emph{(}then $\cS,\cM$ are subsets of  $H^1(\R^3)^K$ instead of  $H^1_{rad}(\R^3)^K$\emph{)}.
\end{Th}

Note that if $\lambda_i$ are prescribed and positive while $\rho_i$ are free, then there are no ground state solutions in $H^1(\R^3)^K$ if $p=4$ and all $\beta_{i,j}<0$, $i\ne j$ \cite{LinWei}. On the other hand, in this case ground states do exist in $H^1_{rad}(\R^3)^K$ \cite{Sirakov}. 

The assumption $p\ge 4$ we have made in Theorems \ref{th:mainK2}, \ref{th:mainK} and \ref{th:PS} is essential for our arguments in Section \ref{sec:PS} (see the application of the Cwikel-Lieb-Rosenblum theorem there). However, we expect the results to be true for all $\frac{10}{3}<p<6$.

\medskip

The paper is organized as follows. In Section \ref{sec:PS} we prove Theorem \ref{th:PS}. Theorem \ref{th:mainK2} is proved in Section \ref{sec:Main2} and Theorem \ref{th:mainK} in Section \ref{sec:Main1}. The latter result will follow from a more general condition specified in Lemma \ref{lem:c_0_general}. In fact also Theorem \ref{th:mainK2} follows from this lemma but for technical simplicity we prefer to prove it separately. Nonexistence of ground states in the repulsive case $\beta_{i,j}<0$ (Theorem \ref{th:nonexistence}) is proved in Section \ref{sec:Main3}. A Liouville-type lemma which is used in Theorem \ref{th:PS} is proved in Appendix \ref{spectrum} in a more general setting than needed since we think it can be interesting in itself.

\medskip

\noindent\textbf{Notation.} $|\cdot|_k$ is the norm in $L^k(\R^3)$, $B(0,\rho)$ is the ball of radius $\rho$ and center at the origin, $\rightharpoonup$ denotes the weak convergence, $V^+ := \max\{V,0\}$, $V^- := \min\{V,0\}$.

\section{Proof of Theorem \ref{th:PS}}\label{sec:PS}

In order to study the compactness of Palais-Smale  sequences of $J$ on $\cS$ we apply the Cwikel-Lieb-Rozenblum theorem  \cite{Cwikel,Lieb,Rozenblum}. We also need a Liouville-type lemma outside a large ball. More precisely,
suppose that the potential $V$ in the  Schr\"odinger operator $-\Delta +V$ on $L^2(\R^3)$ satisfies $\lim_{|x|\to\infty}V(x)|x|^2=0$. If $u\in H^1(\R^3)$  and there is $\rho>0$ such that $u\geq 0$ and $-\Delta u+ V(x)u\geq 0$ for $|x|\geq \rho$, then 
	$\inf\{u(x): |x|=M\}=0$ for any $M\geq \rho$ (so in particular, $u(x)=0$ for $|x|\ge \rho$ if $u\in H^1_{rad}(\R^3)$). A more general version of this result is proved in Appendix \ref{spectrum} since we believe it can be interesting in itself. It extends the results in \cite{BartschSoaveJFA,Ikoma}. 
	
The main novelty here is that we show that regardless of the signs of $\beta_{i,j}$ for $i\ne j$, if
$\lambda_i=0$, then $u_i=0$ also when we a priori do not know whether
$u_i\ge 0$.
	
\medskip

\begin{altproof}{Theorem \ref{th:PS}}
	First of all, note that $M(u^n)\to 0$ implies 
	\begin{equation} \label{eq:bound}
	 \frac{3p-10}{6(p-2)}\int_{\R^3}|\nabla u^n|^2\,dx = J(u^n)-\frac{2}{3(p-2)}M(u^n) \to c > 0
	\end{equation}
	and hence $(u^n)$ is bounded in $H^1_{rad}(\R^3)^K$.
	Therefore we may assume $u^n\weakto u=(u_1,\ldots, u_K)$ in $H^1_{rad}(\R^3)^K$ and $u^n\to u$ in $L^p(\R^3)^K$. 
	Note also that 
	$$
	J'(u^n)+\lambda^n u^n=o(1)
	$$ 
	for some Lagrange multipliers $\lambda^n\in\R^K$ (here $\lambda^nu^n$ should be understood as an element of the dual space of $H^1_{rad}(\R^3)$). It follows that
	\begin{equation}\label{eq:lambda}
	-\lambda^n_i\rho_i^2=\int_{\R^3}|\nabla u^n_i|^2\,dx-\sum_{1\leq j\leq K}\beta_{i,j}\int_{\R^3}|u^n_i|^{p/2}|u^n_j|^{p/2}\,dx+o(1),
	\end{equation}
	so $(\lambda^n)$ is bounded and passing to a subsequence, $\lambda^n\to\lambda$ for some $\lambda\in\R^K$. Observe that 
	$$
	J'(u)+\lambda u=0
	$$ 
	and as a solution, $u$ satisfies the Nehari-Poho\v{z}aev identity \eqref{eq:M}. 
	Thus $M(u)=0$ and therefore
	\begin{equation} \label{eq:conv}
	\int_{\R^3}(|\nabla u^n|^2-|\nabla u|^2)\,dx=\frac{3(p-2)}{2p}\sum_{1\leq i,j\leq K}\beta_{i,j}\int_{\R^3} (|u^n_i|^{p/2}|u_j^n|^{p/2}-|u_i|^{p/2}|u_j|^{p/2})\,dx + o(1) \to 0
	\end{equation}
	as $n\to\infty$.
So $J(u)=c$  and in particular, $u\ne 0$. 
	
	We shall show that if $u_i\ne 0$, then $\lambda_i\ne 0$. Arguing by contradiction, suppose $\lambda_1=0$ and $u_1\ne 0$. Then $\partial_{u_1} J(u)=0$. Let
	$(\Om_k)_{k=1}^{m}$ be the sequence of nodal domains of $u_1$ and let $u_{1,k}:=u_1$ on $\Om_k$ and $u_{1,k}=0$ on $\R^3\setminus \Om_k$. Clearly, $u_{1,k}\in H^1_{rad}(\R^3)$. We claim $m$ must be finite. Indeed, we have
		\begin{align}
	\label{secondderivative}	& \partial^2_{u_1} J(u)[u_{1,k},u_{1,k}]=\int_{\R^3}|\nabla u_{1,k}|^2\,dx-(p-1)\beta_{1,1}\int_{\R^3}|u_{1,k}|^p\,dx \\
	\nonumber	& \hskip8cm -\frac{p-2}2\sum_{j=2}^K\beta_{1,j} \int_{\R^3}|u_{1,k}|^{p/2}|u_j|^{p/2}\,dx\\
	\nonumber	&  < \int_{\R^3}|\nabla u_{1,k}|^2 \,dx- \beta_{1,1} \int_{\R^3}|u_{1,k}|^p\,dx - \sum_{j=2}^K\beta_{1,j}\int_{\R^3}|u_{1,k}|^{p/2}|u_j|^{p/2}\,dx=\partial_{u_1} J(u)[u_{1,k}]=0
	\end{align}
	for $k=1,\ldots,m$
	where the inequality above is a consequence of the identity 
	\begin{eqnarray*}
	&&(p-2)\beta_{1,1}\int_{\R^3}|u_{1,k}|^p\,dx+\frac{p-4}2\sum_{j=2}^K\beta_{1,j}\int_{\R^3}|u_{1,k}|^{p/2}|u_j|^{p/2}\,dx\\
	&& \quad = (p-2)\beta_{1,1}\int_{\R^3}|u_{1,k}|^p\,dx+\frac{p-4}2\sum_{j=2}^K\beta_{1,j}\int_{\R^3}|u_{1,k}|^{p/2}|u_j|^{p/2}\,dx + \frac{p-4}2\partial_{u_1}J(u)[u_{1,k}] \\
	&& \quad = \frac{p-4}2\int_{\R^3} |\nabla u_{1,k}|^2\,dx + \frac p2\int_{\R^3}\beta_{1,1}|u_{1,k}|^p > 0. 
	\end{eqnarray*}
	It follows from \eqref{secondderivative} that $m\leq m(u_1)$ where  $m(u_1)$ stands for the Morse index of the second derivative of
	$$
	H^1_{rad}(\R^3)\ni v\mapsto J(v,u_2,\ldots, u_K)
	$$
	at $v=u_1$. The idea of estimating the number of nodal domains by the Morse index goes back to Benci and Fortunato \cite{benfort}, see also Bahri and Lions \cite{BahriLions}. According to the Cwikel-Lieb-Rozenblum  theorem \cite{Cwikel,Lieb,Rozenblum}, the operator 
	$$Av:=-\Delta v+W(x)v$$ 
	has discrete spectrum on the negative real line and the number of negative eigenvalues is bounded by a constant times $\int_{\R^3}|W^-|^{3/2}\,dx$ provided this integral is finite and $W\in L^1_{loc}(\R^3)$. A convenient reference to this result is \cite[Theorem 4.31]{flw}.
If we put
	$$
	W(x):=-(p-1)\beta_{1,1} |u_1|^{p-2} - \frac{p-2}2\sum_{j=2}^K\beta_{1,j} |u_1|^{p/2-2}|u_j|^{p/2},
	$$
	then $W\in L^{3/2}(\R^3)$ because $2\le\frac32(p-2)<6$. Therefore $m(u_1)$ (which equals the number of negative eigenvalues of $-\Delta+W)$ is finite and hence so is $m$ which proves our claim.
	So there is $\rho>0$ such that $u_1$ has constant sign on $\R^3\setminus B(0,\rho)$
	or $u_1=0$ on $\R^3\setminus B(0,\rho)$. If the latter case occurs, then by the unique continuation property \cite[Theorem 1.4]{GarofaloLin} we obtain $u_1=0$ which is a contradiction. So we may assume $u_1>0$ on $\R^3\setminus B(0,\rho)$. In view of \cite{Strauss}, there is a constant $C>0$ such that $|u_i(x)|\leq C/|x|$ for $i=1,\ldots,K$ and $|x|\geq 1$. In fact, as observed in \cite[Subsection 2.2 (c)]{BerLions},  $u_i(x) = o(|x|)$ as $|x|\to\infty$. This follows from the proof of \cite[Radial Lemma A.II]{BerLions} where one observes that for any $r=|x|>0$ and $v=v(r)\in \cC_{0,rad}^\infty(\R^3)$
	\begin{align*}
	r^2v(r)^2 & = -\int_r^\infty(\rho^2v(\rho)^2)'\,d\rho \le  C\|v\|_{H^1_{rad}(\R^3\setminus B(0,r))}
	\end{align*} 
	($C$ independent of $r$). More precisely, integration in \cite{BerLions} is performed from $0$ to $r$ and $r^2v(r)^2$ is estimated by the norm in $H^1_{rad}(\R^3)$ but the argument is exactly the same in our case.
	Hence setting
	$$V(x):=-\sum_{j=1}^K\beta_{1,j}|u_1|^{p/2-2}|u_j|^{p/2}$$
	we get 
$$\lim_{|x|\to\infty}V(x)|x|^2=0$$
	provided $p\geq 4$. 	
	Since $\partial_{u_1}J(u) = 0$, it follows that 
	$$-\Delta u_1+V(x)u_1=0,$$ 
	and by  Lemma \ref{lemLu} we get $u_1=0$, a contradiction again.  We have shown that $\lambda_1\neq 0$ and in a similar way we show   $\lambda_i\neq 0$ for $i=2,\ldots,K$.

It remains to show that $u_i^n\to u_i$ in $L^2(\R^3)$ for $i=1,\ldots, K$. But as $\lambda_i\ne 0$, using \eqref{eq:conv} we obtain 
	\begin{eqnarray*}
		o(1)&=&(\partial_{u_i}J(u_n)-\partial_{u_i}J(u))[u^n_i-u_i]+\int_{\R^3}(\lambda^n_i u^n_i-\lambda_i u_i)(u^n_i-u_i)\,dx\\
		&=&-\lambda_i\int_{\R^3}|u^n_i-u_i|^2\,dx+o(1)
	\end{eqnarray*}
	and the conclusion follows.
\end{altproof}

\begin{Rem}
\emph{
In fact we have shown a little more: For each $i=1,\ldots,K$ either $\lambda_i=0$ and $u_i=0$, or $\lambda_i\ne 0$ and $u_i^n\to u_i$ in $H^1_{rad}(\R^3)$. In particular, in the latter case $|u_i|_2 = \rho_i$.
}
\end{Rem}

\section{Proof of Theorem \ref{th:mainK2}} \label{sec:Main2}

In order to prove Theorem \ref{th:mainK2} we need some auxiliary results. As they will also be needed in  Theorem \ref{th:mainK}, we formulate them for a general $K$ and not separately for $K=2$.

For $u\in H^1(\R^3)^K\setminus\{0\}$ and $s>0$, define 
\[
s\star u(x):=s^{3/2}u(sx) \quad \text{and} \quad \vp(s):=J(s\star u)
\]
(this definition goes back to \cite{Jeanjean}). Then
\begin{equation} \label{sstar}
\vp(s) = J(s\star u) = \frac{s^2}2\int_{\R^3}|\nabla u|^2\,dx - \frac{s^{\frac{3(p-2)}2}}p\sum_{1\leq i, j\leq K}\beta_{i,j} \int_{\R^3}|u_i|^{p/2}|u_j|^{p/2}\,dx.
\end{equation}
Recall the following lemma.

\begin{Lem}[\cite{MederskiSchino,JeanjeanLuNorm}]\label{lem:phi}
	If $u\in H^1_{rad}(\R^3)^K\setminus\{0\}$, then
	there exists a global maximizer  $s_u>0$ for $\vp$ and $\vp$ is strictly increasing on $(0,s_u)$ and strictly decreasing on $(s_u,+\infty)$. Moreover, $s\star u\in\cM$ if and only if $s=s_u$, $M(s\star u)>0$ if and only if $s\in(0,s_u)$
	and $M(s\star u)<0$ if and only if $s>s_u$. 
\end{Lem}

The proof is by a simple computation using \eqref{sstar}; in \cite{MederskiSchino,JeanjeanLuNorm} this result has been proved in a more general situation. 
It follows from Lemma \ref{lem:phi} that $\cS\cap\cM\neq\emptyset$. We show that  $\cS$ and $\cM$ are $\cC^2$-manifolds which intersect transversally. That $\cS$ is of class $\cC^2$ is obvious. It is easily seen from \eqref{eq:M} that $M$ is a $\cC^2$-functional and $M'(u)[u]\ne 0$ for $u\in\cM$. Hence $\cM$ is a $\cC^2$-manifold.  We need to show that  $\partial_{u_i}M(u)$ and $u_i$ are linearly independent for $i=1,2,\ldots,K$. Assuming the contrary, we have
$$
\partial_{u_i}M(u)+\lambda_i u_i=0\hbox{ for some }\lambda_i\in \R, \ i=1,\ldots,K
$$ 
and 
\begin{align*}
\int_{\R^3} |\nabla u|^2 \, dx &=  \frac{3(p-2)}{2p}\,  \frac{3(p-2)}{2}\Big(\sum_{1\leq i,j\leq K}\beta_{i,j} \int_{\R^3}|u_i|^{p/2}|u_j|^{p/2}\,dx\Big)\\
	& > \frac{3(p-2)}{2p} \Big(\sum_{1\leq i,j\leq K}\beta_{i,j} \int_{\R^3}|u_i|^{p/2}|u_j|^{p/2}\,dx\Big) \\
	&  = \int_{\R^3}|\nabla u|^2\,dx-M(u)=\int_{\R^3}|\nabla u|^2\,dx
\end{align*}
where the first equality above is the Nehari-Poho\v{z}aev identity for $M$.
So we obtain a contradiction which shows that  $\cS$ and $\cM$ intersect transversally and $\cS\cap\cM$ is a $\cC^2$-manifold of codimension $K+1$ in $ H^1_{rad}(\R^3)^K$.

For a closed and symmetric set $A$ in a Banach space  we define the genus $\gamma(A)$ as the smallest integer $n$ such that there exists a continuous and odd map $h:A\to\R^n\setminus\{0\}$. If no such $h$ exists, we set $\gamma(A):=\infty$, and we also define $\gamma(\emptyset)=0$. Basic properties of genus may be found e.g. in \cite{Struwe}.
In the proposition below we compute the genus of a product of spheres. The result may be known but we could not find any convenient reference to it. Let $S^{k-1}$ denote the unit sphere in $\R^k$.

\begin{Prop} \label{prodspheres}
	Let $S:= S^{m_1-1}\times\cdots\times S^{m_k-1}$. Then $\gamma(S) = \min\{m_1,\ldots,m_k\}$.
\end{Prop}

\begin{proof}
	Assume without loss of generality that $m_1\le m_j$ for $2\le j\le k$. The map 
	\[
	h: S\to S^{m_1-1}, \qquad h(x_1,\ldots,x_k) = x_1
	\]
	is odd, hence $\gamma(S) \le \gamma(S^{m_1-1}) = m_1$. Since also
	\[
	g: S^{m_1-1} \to S, \qquad g(x_1) = (x_1,\ldots, x_1)
	\]
	is odd, $\gamma(S^{m_1-1})\le \gamma(S)$.
\end{proof}

From now on and to the end of this section we assume $K=2$. 
Let
\[
\Sigma := \{A\subset \cS\cap\cM : A=-A\text{ and $A$ is compact}\}
\]
and
\begin{equation} \label{eq:cm}
c_m(\beta) := \inf_{\substack{A\in\Sigma\\\gamma(A)\ge m}}\max_{u\in A} J(u)
\end{equation}
where $\beta:= \beta_{1,2}=\beta_{2,1}$.

Consider the equation
\begin{equation} \label{eq:ui}
\left\{ \begin{array}{l}
-\Delta u_i+ \lambda_i u_i = \beta_{i,i}|u_i|^{p-2}u_i \quad \mathrm{in} \ \R^3, \\
u_i \in H^1_{rad}(\R^3), \\
\int_{\R^3} |u_i|^2 \, dx = \rho_i^2
\end{array} \right. 
\end{equation}
for $i=1,2$. As we have mentioned in the introduction, by a result of Kwong \cite{Kwong} this equation has a unique positive solution and it is known to be a ground state. Denote the smaller of the ground state energies for $i=1,2$  by $c_0$. Next we show that there exists $\beta_m>0$ such that if $\beta > \beta_m$, then $c_m(\beta)<c_0$. This is an immediate consequence of the following result.

\begin{Lem} \label{to0}
$\lim_{\beta\to \infty} c_m(\beta) = 0$.
\end{Lem}

\begin{proof}
Let $A_1\subset H^1_{rad}(\R^3)$ be an $(m-1)$-dimensional sphere such that $|u_1|_2 = \rho_1$ for all $u_1\in A_1$ and let $\rho:= \rho_2/\rho_1$. Then $A:= A_1\times\rho A_1\subset\cS$  and $\gamma(A)=m$ according to Proposition \ref{prodspheres}.  For $u=(u_1,\rho u_1)\in A$ we have
\[
J(s\star u) \le \frac{(1+\rho^2)s^2}{2}\int_{\R^3}|\nabla u_1|^2\,dx - \beta\rho^{p/2}s^{3(p-2)/2}\int_{\R^3}|u_1|^{p}\,dx
\]
(cf. \eqref{sstar}). Since the first integral on the right-hand side above is bounded and the second integral is bounded away from 0 on $A_1$, a simple computation shows that $J(s\star u)\le C/\beta^\alpha$ for some $C>0$,  $\alpha = 4/(3p-10)$ and all $u\in A$, $s>0$. Note that $\alpha>0$ because $p>10/3$. Let $\overline A := \{s_u\star u: u\in A\}$ where $s_u$ is as in Lemma \ref{lem:phi}. Then $\overline A\in \Sigma$ and  $\gamma(\overline A) = \gamma(A)=m$, so $c_m(\beta) \le \max_{v\in\overline A}J(v)\to 0$ as $\beta\to\infty$ and the conclusion follows. 
\end{proof}

\begin{altproof}{Theorem \ref{th:mainK2}}
Choose $\beta_m$ such that $c_m(\beta)<c_0$ for $\beta>\beta_m$. 
Introduce a map $\wt J: \cS\to \R$ by setting
 $\wt J(u):=J(s_u\star u)$. Then
 $$
 \cS\ni u\mapsto s_u\star u\in \cS\cap\cM
 $$ 
 is a homeomorphism, and in a similar way as in \cite[Lemma 3.8]{BartschSoaveCalV}, 
 $\wt J$ is of class $\cC^1$ with 
 \begin{equation} \label{derivative}
 \wt J'(u)[v]=J'(s_u\star u)[s_u\star v]
 \end{equation}
 for every $u\in \cS$ and $v\in T_u\cS$ where $T_u\cS$ is the tangent space of $\cS$ at $u$. Now observe that
 \[
c_j(\beta) = \inf_{\substack{A\in\wt\Sigma\\ \gamma(A)\ge j}}\max_{u\in A} \wt J(u), \quad 1\le j\le m, 
 \]
 where 
 $$\wt 
 \Sigma := \{A\subset \cS  : A=-A\text{ and $A$ is compact}\}
 $$
 and  $u$ is a critical point of $\wt J$ if and only if $s_u\star u$ is a critical point of $J|_\cS$. 

Set $c_j:=c_j(\beta)$. According to \cite[Lemma 2.5]{MederskiSchino}, $c_1>0$ (this can also be proved by an easy adaptation of the proof of Lemma 4.2 below).  Let
\begin{gather*}
\wt J_a := \{u\in\cS: \wt J(u)\ge a\}, \quad \wt J^b := \{u\in\cS: \wt J(u)\le b\}, \quad \wt J_a^b := \wt J_a\cap \wt J^b, \\
K_{c_j} := \{u\in \cS: J(u)=c_j, \ J'|_{\cS}(u)=0\} \quad \text{and} \quad N_\delta(A) := \{u\in \cS: \text{dist}(u, A)<\delta\},
\end{gather*} 
where \text{dist} denotes the distance (we emphasize that in the definition of $K_{c_j}$ the functional is $J$ and not $\wt J$). If $u\in K_{c_j}$, then $u_1, u_2\ne 0$ because $0<c_j<c_0$. Hence it follows from Theorem \ref{th:PS} that $K_{c_j}$ is compact. Let $\delta>0$ be such that $\gamma(N_{2\delta}(K_{c_j})) = \gamma(K_{c_j})$. We claim that there exists $\eps_0>0$ such that if $u\in (\wt J_{c_j-\eps_0}^{c_j+\eps_0}\cap N_{\eps_0}(\cM))\setminus N_{\delta}(K_{c_j})$, then $\|\wt J'(u)\|\ge\eps_0$. Arguing by contradiction, we can find $u^n\notin  N_{\delta}(K_{c_j})$ such that $\wt J(u^n)\to c_j$, $\|u^n-z^n\|\to 0$ for some $z^n\in\cM$ and $\wt J'(u^n)\to 0$, where $\|\cdot\|$ denotes the norm in $H^1(\R^3)^2$. As $M(z^n) = 0$, $(z^n)$ and $(u^n)$ are bounded according to \eqref{eq:bound}. Since by \eqref{sstar},
\begin{equation*}
s_{z^n}^{\frac{3p-10}{2}}=\frac{|\nabla  z^n|_2^2}{ \frac{3(p-2)}{2p}\sum_{1\leq i, j\leq K}\beta_{i,j} \int_{\R^3}|z^n_i|^{p/2}|z^n_j|^{p/2}\,dx} = 1,
\end{equation*}
it is easy to see that $s_{u^n}\to 1$. Indeed, passing to a subsequence, $z^n\rightharpoonup z$ in $H^1_{rad}(\R^3)^K$ and $z^n\to z$ in $L^p(\R^3)^K$. We have 
\[
|\nabla z^n|_2 - |\nabla (u^n-z^n)|_2 \le  |\nabla u^n|_2 \le |\nabla z^n|_2 + |\nabla (u^n-z^n)|_2
\]
and hence
\begin{equation*}
s_{u^n}^{\frac{3p-10}{2}}=\frac{|\nabla  z^n|_2^2 + o(1)}{ \frac{3(p-2)}{2p}\sum_{1\leq i, j\leq K}\beta_{i,j} \int_{\R^3}|z^n_i|^{p/2}|z^n_j|^{p/2}\,dx + o(1)} \to 1.
\end{equation*}
Note that here $K=2$ and $\beta_{ij}=\beta$ but we shall need the equality above for a general $K$ in the next section.
Denote $v^n := s_{u^n}\star u^n$. Then $J(v^n)\to c_j$ and by \eqref{derivative}, $J|_\cS'(v^n)\to 0$.  Passing to a subsequence, $v^n\rightharpoonup v$ and according to Theorem \ref{th:PS}, either $v^n\to v\ne 0$ or one of the components of $v$, say $v_1$, equals zero. But then $J(0,v_2)=c_j$ and $v_2$ is a solution to \eqref{eq:ui}. This is impossible because $c_j<c_0$. Hence $v^n\to v$ and therefore $u^n\to u\in K_{c_j}$. Since $u\notin N_{\delta/2}(K_{c_j})$, we obtain a contradiction and the claim follows. 

Take $\eps\in(0,\eps_0/2)$ and let $\eta:\cS\times[0,1]\to\cS$ be the pseudo-gradient flow constructed in the deformation lemma \cite[Lemma 5.15]{Willem}. However, we need to make a small change in the definition of the sets $\cA$ and $\cB$ used in \cite{Willem} in the construction of the cutoff function $\psi$ related to the vector field $g$. Here we set $\cA := (\wt J_{c_j-2\eps}^{c_j+2\eps}\cap \cN_{\eps_0}(\cM))\setminus N_{\delta}(K_{c_j})$, $\cB := (\wt J_{c_j-\eps}^{c_j+\eps}\cap \cN_{\eps_0/2}(\cM))\setminus N_{2\delta}(K_{c_j})$ and we choose $\eps<\eps_0/2$ so that $8\eps/\delta<\eps_0$. The rest of the proof in \cite{Willem} remains unchanged and we get
\begin{equation} \label{deform}
\eta((\wt J^{c_j+\eps}\cap \cM)\setminus N_{2\delta}(K_{c_j}),1)\subset \wt J^{c_j-\eps}.
\end{equation}
We may also assume the vector field $g$, and hence $\eta$, is odd in $u$.

Now, let $A \in \wt\Sigma$ be such that $\gamma(A)\ge j$ and $\sup_{u\in A}\wt J(u) \le c_j+\eps$. Set $B := \{s_u\star u: u\in A\}$. Then $B\in\wt\Sigma$ and $\gamma(B)\ge j$. Moreover, $B\subset \cM$ and since $\wt J(u) = \wt J(s_u\star u)$, $\sup_{u\in B}\wt J(u) \le c_j+\eps$. By \eqref{deform}, $\eta(B\setminus N_{2\delta}(K_{c_j}),1)\subset \wt J^{c-\eps}$. Suppose $c_j=\cdots = c_{j+p}$ for some $p\ge 0$. Then by standard arguments, see e.g. \cite[Lemma II.5.6 and Theorem II.5.7]{Struwe}, $\gamma(K_{c_j})\ge p+1$. In particular, $K_{c_j}\ne\emptyset$.  If all $c_j$ are different, then to each $c_j$ there corresponds a solution $u^j$ (or an orbit of solutions to be more precise) and $J(u^1)<J(u^2)<\cdots< J(u^m)$. If $p>0$ for some $j$, then $\gamma(K_{c_j})\ge 2$. Hence $K_{c_j}$ is infinite and so is the number of (orbits of)  solutions. 

Finally, since $c_1(\beta) = \inf_{u\in\cS\cap\cM} J(u)$, we may replace the minimizer $u=(u_1,u_2)$ by $\bar u=(|u_1|,|u_2|)$. Then $\bar u$ is a ground state with positive components.
\end{altproof}

\section{Proof of Theorem \ref{th:mainK}}\label{sec:Main1}

Let $C_{p}>0$ denote the optimal constant in the {\em Gagliardo-Nirenberg inequality}
\begin{equation}\label{eq:GN}
|u|_p \leq C_{p} |\nabla u|_2^{\delta_p} |u|_2^{1-\delta_p}\quad\hbox{for }u\in H^1(\R^3),
\end{equation}
$\delta_p = 3 \big( \frac{1}{2} - \frac{1}{p} \big)$ and $\delta_pp>2$. 

Define
\begin{equation} \label{czero}
	c_0:=\min\big\{\inf_{\cM_\cI\cap\cS_\cI} J :\; \cI\subset\{1,\ldots,K\},\;1\leq |\cI|\leq K-1   \big\},
\end{equation}
where $|\cI|$ is the cardinality of $\cI$ and
\begin{eqnarray*}
	\cS_\cI&:=& \left\{ u \in H^1_{rad}(\R^3)^K  : \int_{\R^3} |u_i|^2 \, dx = \rho_i^2 \text{ if } i\in\cI\hbox{ and }u_i=0\text{ if } i\notin\cI\right\},\\
	\cM_\cI&:=& \left\{ u \in \cM:  u_i=0 \text{ if } i\notin\cI \right\}.
\end{eqnarray*}
Note that 
$$\inf_{\cM_\cI\cap\cS_\cI} J>0$$
for each $ \cI\subset\{1,\ldots,K\},\;1\leq |\cI|\leq K-1$ and
 $c_0>0$, see e.g. \cite[Lemma 2.3, Lemma 2.5]{MederskiSchino}.

Next we want to obtain a lower bound of $\inf_{\cM_\cI\cap\cS_\cI} J$.  In order to do this,  we need  the following inequalities.

\begin{Lem}\label{eq:nequi} For any $\alpha\geq 2$, $q\geq 1$, $m\ge 2$ and $a_1,a_2,\ldots,a_m\in\R^+$ there holds 
	\begin{eqnarray}\label{eq:nequi1}
		\Big(\sum_{i=1}^m a_i^{\alpha}\Big)^{1/\alpha}&\leq& 	\Big(\sum_{i=1}^m a_i^{2}\Big)^{1/2},\\
		\label{eq:nequi2}	\Big(\frac{1}{m(m-1)}\sum_{1\leq i\neq j\leq m} a_ia_j\Big)^{1/2}&\leq& \Big(\frac{1}{m}\sum_{i=1}^m a_i^{q}\Big)^{1/q}.
	\end{eqnarray}
	\end{Lem}
	
\begin{proof}
Inequality \eqref{eq:nequi1} is a special case of \cite[(2.10.3)]{Hardy} (it holds for any $\alpha\ge \beta>0$ where $\beta$ replaces 2 on the right-hand side).

 In order to prove \eqref{eq:nequi2}, note that
$$
\sum_{1\leq i\neq j\leq m} a_ia_j\leq
(m-1)\sum_{i=1}^m a_i^2=(m-1)\Big(\sum_{i=1}^m a_i\Big)^2-(m-1)\sum_{1\leq i\neq j\leq m} a_ia_j.
$$
Hence
\[
\sum_{1\leq i\neq j\leq m}a_ia_j \le \frac{m-1}m\Big(\sum_{i=1}^m a_i\Big)^2
\]
and
\[
\sum_{1\leq i\neq j\leq m} a_ia_j\leq m(m-1)
\Big(\frac{1}{m}\sum_{i=1}^m a_i\Big)^{2}
\leq m(m-1)
\Big(\frac1m\sum_{i=1}^m a_i^{q}\Big)^{2/q}
\]
where the last inequality is a special case of \cite[(2.10.4)]{Hardy}. 
\end{proof}

\begin{Lem}\label{lem32}Let  $ \cI\subset\{1,\ldots,K\},\;1\leq |\cI|\leq K-1 $. There holds
$$
\inf_{\cM_\cI\cap\cS_\cI} J\geq 
\frac{3p-10}{6(p-2)} \Big[
\frac{3(p-2)}{2p}C_p^{p}\Big(\max_{i\in\cI}\{\beta_{i,i}\rho_i^{\frac{6-p}2}\}+\frac{|\cI|-1}{|\cI|^{\frac{3p-10}{4}}}\max_{i,j\in\cI,i\neq j}\{\beta_{i,j}(\rho_i\rho_j)^{\frac{6-p}4}\}\Big)\Big]^{-\frac{4}{3p-10}}.
$$
\end{Lem}

\begin{proof}
By the H\"older inequality,
\[
\sum_{i,j\in \cI}\beta_{i,j}\int_{\R^3} |u_i|^{p/2}|u_j|^{p/2} \,dx \le  \sum_{i\in\cI}\beta_{i,i}|u_i|_p^p + \sum_{i,j\in \cI, i\ne j}\beta_{i,j}|u_i|_p^{p/2}|u_j|_p^{p/2}.
\]
Taking into account \eqref{eq:GN} and using \eqref{eq:nequi1} with $a_i = |\nabla u_i|_2$, $\alpha=\frac32(p-2)$ in the first term on the right-hand side above and \eqref{eq:nequi2} with $a_i=|\nabla u_i|_2^{\frac34(p-2)}$, $m=|\cI|$, $q=\frac{8}{3(p-2)}$ in the second term we get the  inequality 
\begin{eqnarray*}
&&\sum_{i,j\in\cI}\beta_{ij}\int_{\R^3} |u_i|^{p/2}|u_j|^{p/2} \,dx\\&&\leq\nonumber
C_p^{p}\Big(\max_{i\in\cI}\{\beta_{i,i}\rho_i^{\frac{6-p}2}\}|\nabla u|_2^{\frac32(p-2)}+\frac{|\cI|-1}{|\cI|^{\frac{3p-10}{4}}}\max_{i,j\in\cI,i\neq j}\{\beta_{i,j}(\rho_i\rho_j)^{\frac{6-p}4}\}|\nabla u|_2^{\frac32(p-2)}\Big)
\end{eqnarray*}
for  $u\in\cS_\cI$. Note that $p\le \frac{14}3$ is needed here in order to have $q\ge 1$.
Using \eqref{eq:M} we obtain
\begin{eqnarray*}
|\nabla u|_2^2&\leq &\frac{3(p-2)}{2p}C_p^{p}\Big(\max_{i\in\cI}\{\beta_{i,i}\rho_i^{\frac{6-p}2}\}+\frac{|\cI|-1}{|\cI|^{\frac{3p-10}{4}}}\max_{i,j\in\cI,i\neq j}\{\beta_{i,j}(\rho_i\rho_j)^{\frac{6-p}4}\}\Big)|\nabla u|_2^{\frac32(p-2)}
\end{eqnarray*}
for $u\in\cM_\cI\cap\cS_\cI$. Since $M(u)=0$, it follows using \eqref{eq:bound} that
\begin{eqnarray*}
J(u)&=&\frac{3p-10}{6(p-2)}|\nabla u|_2^2\\
&\geq& \frac{3p-10}{6(p-2)} \Big[
\frac{3(p-2)}{2p}C_p^{p}\Big(\max_{i\in\cI}\{\beta_{i,i}\rho_i^{\frac{6-p}2}\}+\frac{|\cI|-1}{|\cI|^{\frac{3p-10}{4}}}\max_{i,j\in\cI,i\neq j}\{\beta_{i,j}(\rho_i\rho_j)^{\frac{6-p}4}\}\Big)\Big]^{-\frac{4}{3p-10}}
\end{eqnarray*}
as claimed.
\end{proof}

As in the preceding section, let
\[
\Sigma := \{A\subset \cM\cap\cS : A=-A\text{ and $A$ is compact}\}.
\]
We shall show that 
\[
c_m(\boldsymbol\beta) := \inf_{\substack{A\in\Sigma\\\gamma(A)\ge m}}\max_{u\in A} J(u)<c_0
\]
for suitable $\boldsymbol\beta := \{\beta_{i,j}: 1\le i,j\le K\}$.

Let  $\cS_m$ denote the collection of all $(m-1)$-dimensional unit spheres in $L^2_{rad}(\R^3)$. We set
$$
\Theta_m:=\inf_{A\in\cS_m}\sup_{u\in A}\frac{\big(\int_{\R^3}|\nabla u|^2\,dx\big)^{\frac{3(p-2)}{3p-10}}}{\big(\int_{\R^3}|u|^p\,dx\big)^{\frac{4}{3p-10}}}.
$$

\begin{Lem} \label{lem:c_0_general} 
For any $m\geq 1$,
	$c_m(\boldsymbol\beta) < c_0$ provided that
\begin{align} \label{eq:c_0}
	&\max_{\substack{\cI\subset\{1,\ldots,K\}\\1\leq |\cI|\leq K-1}}\Big(\max_{i\in\cI}\{\beta_{i,i}\rho_i^{\frac{6-p}2}\}+\frac{|\cI|-1}{|\cI|^{\frac{3p-10}{4}}}\max_{i,j\in\cI,i\neq j}\{\beta_{i,j}(\rho_i\rho_j)^{\frac{6-p}4}\}\Big)\frac{\big(\sum_{i=1}^K\rho_i^2\big)^{\frac34(p-2)}}{\sum_{1\leq i, j\leq K}\beta_{i,j}(\rho_i\rho_j)^{p/2}} \\
\nonumber	&\qquad<\Theta_m^{-\frac{3p-10}{4}}C_p^{-p}.
\end{align}
Moreover, $\Theta_m^{-\frac{3p-10}{4}}C_p^{-p}\leq 1$ for all $m\geq 1$ and $\Theta_1^{-\frac{3p-10}{4}}C_p^{-p}=1$.
\end{Lem}

\begin{proof}
Let $A_1\in\cS_m$ and $\widehat{A}:= \{(\rho_1v, \rho_2v, \ldots, \rho_Kv): v\in A_1\}\subset \cS$.  Since $\widehat{A}$ is homeomorphic to $A_1$, $\gamma(\widehat A)=m$. 
In view of  Lemma \ref{lem:phi}  we may define
	$$
	\widetilde A := \{s_u\star u:\; u\in \widehat A\}.
	$$
	Observe that $\widetilde A\in\Sigma$ and $\gamma(\widetilde A) = m$.
	For $v\in A_1$ we take $u:=(\rho_1 v,\ldots, \rho_K v)\in \widehat A$ and $s_u\star u\in\widetilde A$. By \eqref{sstar}, 
	$$
	s_u^{\frac{3p-10}{2}}=\frac{|\nabla  u|_2^2}{ \frac{3(p-2)}{2p}\sum_{1\leq i, j\leq K}\beta_{i,j} \int_{\R^3}|u_i|^{p/2}|u_j|^{p/2}\,dx}
	$$
Hence using \eqref{eq:bound} and $M(u)=0$, 
	\begin{eqnarray*}
		J(s_u\star u)
		&=&\frac{3p-10}{6(p-2)} s_u^2|\nabla u|_2^2=\frac{3p-10}{6(p-2)}\, \frac{|\nabla u|_2^{\frac{6(p-2)}{3p-10}}}{\big( \frac{3(p-2)}{2p}.\sum_{1\leq i, j\leq K}\beta_{i,j} \int_{\R^3}|u_i|^{p/2}|u_j|^{p/2}\,dx\big)^{\frac{4}{3p-10}}}\\
		&=&\frac{3p-10}{6(p-2)}\, \frac{\big(\sum_{i=1}^K\rho_i^2\big)^{\frac{3(p-2)}{3p-10}}|\nabla v|_2^{\frac{6(p-2)}{3p-10}}}{\big( \frac{3(p-2)}{2p}\big)^{\frac{4}{3p-10}}\big(\sum_{1\leq i, j\leq K}\beta_{i,j}(\rho_i\rho_j)^{p/2}\big)^{\frac{4}{3p-10}}\big( \int_{\R^3}|v|^{p}\,dx\big)^{\frac{4}{3p-10}}}
	\end{eqnarray*}
and
	\begin{eqnarray*}
	c_m(\boldsymbol\beta) &\leq&
	\sup_{u\in \widetilde A} J(u)= \sup_{u\in\widehat A} J(s_u\star u) \\
& = & \sup_{v\in A_1}\frac{\big(\int_{\R^3}|\nabla v|^2\,dx\big)^{\frac{3(p-2)}{3p-10}}}{\big(\int_{\R^3}|v|^p\,dx\big)^{\frac{4}{3p-10}}} \,\frac{3p-10}{6(p-2)}\, \frac{\big(\sum_{i=1}^K\rho_i^2\big)^{\frac{3(p-2)}{3p-10}}}{\big( \frac{3(p-2)}{2p}\big)^{\frac{4}{3p-10}}\big(\sum_{1\leq i, j\leq K}\beta_{i,j}(\rho_i\rho_j)^{p/2}\big)^{\frac{4}{3p-10}}}.
\end{eqnarray*}
Since this holds for all $A_1\in\cS_m$, it follows that 
$$
c_m(\boldsymbol\beta)\leq \Theta_m\, \frac{3p-10}{6(p-2)}\, \frac{\big(\sum_{i=1}^K\rho_i^2\big)^{\frac{3(p-2)}{3p-10}}}{\big( \frac{3(p-2)}{2p}\big)^{\frac{4}{3p-10}}\big(\sum_{1\leq i, j\leq K}\beta_{i,j}(\rho_i\rho_j)^{p/2}\big)^{\frac{4}{3p-10}}}.
$$
By \eqref{eq:c_0},
\[
\Big[C_p^{p}\Big(\max_{i\in\cI}\{\beta_{i,i}\rho_i^{\frac{6-p}2}\}+\frac{|\cI|-1}{|\cI|^{\frac{3p-10}{4}}}\max_{i,j\in\cI,i\neq j}\{\beta_{i,j}(\rho_i\rho_j)^{\frac{6-p}4}\}\Big)\Big]^{-\frac{4}{3p-10}} > \Theta_m \frac{\big(\sum_{i=1}^K\rho_i^2\big)^{\frac{3(p-2)}{3p-10}}}{\big(\sum_{1\leq i, j\leq K}\beta_{i,j}(\rho_i\rho_j)^{p/2}\big)^\frac4{3p-10}},
\]
and in view of Lemma \ref{lem32},
\[
\inf_{\cM_\cI\cap\cS_\cI} J > \Theta_m \frac{3p-10}{6(p-2)} \frac{\big(\sum_{i=1}^K\rho_i^2\big)^{\frac{3(p-2)}{3p-10}}}{\big( \frac{3(p-2)}{2p}\big)^{\frac{4}{3p-10}}\big(\sum_{1\leq i, j\leq K}\beta_{i,j}(\rho_i\rho_j)^{p/2}\big)^{\frac{4}{3p-10}}} \ge c_m(\boldsymbol\beta).
\]
for each $ \cI\subset\{1,\ldots,K\},\;1\leq |\cI|\leq K-1$. Hence, recalling the definition \eqref{czero}, $c_m(\boldsymbol\beta) < c_0$.

Observe that $\Theta_m^{-\frac{3p-10}{4}}C_p^{-p}\leq 1$ follows from the Gagliardo-Nirenberg inequality \eqref{eq:GN}.  If $m=1$, then we may take
$$A_1:=\{u,-u\},$$
where $u\in H^1_{rad}(\R^3)$ is an optimizer for \eqref{eq:GN} and $|u|_2=1$. That such $u$ exists is a consequence of \cite[Corollaries 2.1 and 2.2]{Weinstein}. This shows that 
$\Theta_1=C_p^{-\frac{4p}{3p-10}}$.	
\end{proof}

\begin{altproof}{Theorem \ref{th:mainK}}
Since $\Theta_1^{-\frac{3p-10}4}C_p^{-p}=1$, \eqref{eq:c_0} is equivalent to \eqref{eq:betacond}. Hence for $\boldsymbol\beta$ such that  \eqref{eq:betacond} holds, we have $c_1(\boldsymbol\beta)<c_0$. Now the argument is exactly the same as in the proof of Theorem \ref{th:mainK2} (in fact a little simpler because only $c_1(\boldsymbol\beta)$ is of interest here).
\end{altproof}

\begin{Rem}
\emph{
As we have already mentioned in the introduction, if $K=2$ and $m\ge 1$, then \eqref{eq:c_0} holds for all $\beta$ large enough ($\beta>\beta_m$) because $|\cI|-1=0$. Looking into the proof of Lemma \ref{lem32} we also see that the restriction $p\le\frac{14}3$ is not necessary in this case. Hence we have reproved Theorem \ref{th:mainK2}. However, since our argument in Section \ref{sec:Main2} is considerably simpler, we preferred to treat the case $K=2$ separately. If $K>2$ and \eqref{eq:c_0} holds for some $m$, then the system \eqref{eq} either has  infinitely many solutions or at least $m$ solutions at different energy levels. However, we do not know whether \eqref{eq:c_0} can hold for $K>2$ and $m>1$.
}
\end{Rem}

\section{Nonexistence of ground states}\label{sec:Main3}

Denote $J_i(v) := J(0,\ldots,v,\ldots,0)$ and 
	\begin{eqnarray*}
	\cS_{i}&:=&\Big\{u\in H^1_{rad}(\R^3):\; \int_{\R^3}|u|^2\,dx=\rho_i^2\Big\},\\
		\cM_i&:=&\big\{u\in H^1_{rad}(\R^3): M_i(u) := M_i(0,\dots,u,\ldots,0)=0\big\}
	\end{eqnarray*}
for $i=1,\ldots,K$.

\begin{altproof}{Theorem \ref{th:nonexistence}}
We prove the nonexistence of radial ground states. The proof in the non-radial case is exactly the same except that in the definitions of $\cS,\cM,\cS_i,\cM_i$ one needs to replace $H^1_{rad}(\R^3)$ with $H^1(\R^3)$. 

 As we have mentioned earlier, by the uniqueness result of Kwong \cite{Kwong} the equation
 \begin{equation*}
\left\{ \begin{array}{l}
-\Delta u + \lambda u = \beta_{i,i}|u|^{p-2}u, \\
\int_{\R^3} u^2 \, dx = \rho_i^2
\end{array} \right.
\end{equation*}
has a unique positive ground state solution $u_i:=u_{\beta_{i,i}}$ which is in $H^1_{rad}(\R^3)$. Thus 
$$
J(u_i)=\inf_{\cM_i\cap \cS_{i}}J_i.
$$
Denote $c_i=J_i(u_i)$ and assume without loss of generality that $c_1\le c_i$ for all $i$. We claim that $c_1$ is the ground state level for \eqref{eq} and that it is not attained. To prove this we first show that no level $c\le c_1$ is attained by $J$ on $\cM\cap\cS$ and then exhibit a sequence $(v^n) \subset \cM\cap\cS$ such that $J(v^n)\to c_1$.

Choose $w=(w_1,\ldots,w_K)\in \cM\cap\cS$. Then (cf. \eqref{eq:M})
\[
0 = M(w) \ge M_1(w_1) + \cdots + M_K(w_K)
\]
because $\beta_{i,j}<0$ for all $i\ne j$.
Hence at least one of $M_i(w_i)$ is $\le 0$. Using \eqref{eq:bound}, $M(w)=0$ and the equality $\int_{\R^3}|\nabla(s\star w)|^2\,dx = s^2\int_{\R^3}|\nabla w|^2\,dx$, we obtain 
\[
J(s_w\star w) = \frac{3p-10}{6(p-2)}s_w^2 \int_{\R^3} |\nabla w|^2\,dx = \frac{3p-10}{6(p-2)}\int_{\R^3} (|\nabla w_1|^2+\cdots+|\nabla w_K|^2)\,dx
\]
($s_w=1$ because $w\in\cM\cap\cS$). Likewise,
\[
J_i(s_{w_i}\star w_i) = \frac{3p-10}{6(p-2)} s_{w_i}^2\int_{\R^3} |\nabla w_i|^2\,dx, \quad i=1,\ldots, K.
\]
Note that in order that $s_w$ and $s_{w_i}$ exist it is necessary to have $p>\frac{10}3$.
Choose $i$ with $M_i(w_i)\le 0$. Since $M_i(s\star w_i) >0$ for $s<s_{w_i}$ and $<0$ for $s>s_{w_i}$, we have $s_{w_i}\le 1$ and since $s_{w_i}\star w_i\in \cM_i\cap \cS_i$,  
\begin{align*}
J(w) & = \frac{3p-10}{6(p-2)}\int_{\R^3} |\nabla w|^2\,dx > \frac{3p-10}{6(p-2)}\int_{\R^3} |\nabla w_i|^2\,dx \\ 
& \ge \frac{3p-10}{6(p-2)} s_{w_i}^2\int_{\R^3} |\nabla w_i|^2\,dx \ge c_i \ge c_1.
\end{align*}
So no level $c\le c_1$ is attained by $J$ on $\cM\cap\cS$.

Let $u^n := (u_1,s_n\star u_2,\ldots,s_n\star u_K)$ where $s_n\to 0^+$. Then $u^n\in\cS$, $|\nabla (s_n\star u_i)|_2\to 0$ and $|s_n\star u_i|_p\to 0$ for $2\le i\le K$. Set $u_1^n:=u_1$ and  $u_i^n := s_n\star u_i$ for $2\le i\le K$. Since $|u_i^n|_p\to 0$ for $2\le i\le K$, 
\[
\sum_{1\leq i, j\leq K}\beta_{i,j} \int_{\R^3}|u_i^n|^{p/2}|u_j^n|^{p/2}\,dx >0
\]
for $n$ large enough. For such $n$, $s_{u^n}$ exists, $s_{u^n}\star u^n\in \cM\cap\cS$  and using \eqref{sstar} we obtain
\[
s_{u^n}^{\frac{3p-10}2} = \frac{2p}{3(p-2)}\,\frac{\int_{\R^3}|\nabla u^n|^2\,dx}{\sum_{1\leq i, j\leq K}\beta_{i,j} \int_{\R^3}|u_i^n|^{p/2}|u_j^n|^{p/2}\,dx} \to \frac{2p}{3(p-2)}\, \frac{\int_{\R^3}|\nabla u_1|^2\,dx}{\beta_{1,1}\int_{\R^3}|u_1|^p\,dx} = s_{u_1}^{\frac{3p-10}2} = 1
\]
because $u_1\in \cM_{1}\cap\cS_1$. It follows that 
\[
J(s_{u^n}\star u^n) = \frac{3p-10}{6(p-2)}s_{u^n}^2\int_{\R^3} |\nabla u^n|^2\,dx \to \frac{3p-10}{6(p-2)}\int_{\R^3}|\nabla u_1|^2\,dx = c_1.
\]
This completes the proof.
\end{altproof}

\appendix

\section{A Liouville-type lemma}\label{spectrum}

Let $V\in L^\infty_{loc}(\R^N\setminus B(0,\rho))$ and let $A:=-\Delta +V$ be a Schr\"odinger operator on $L^2\big(\R^N\setminus \overline{B(0,\rho)}\big)$. We assume $V$ satisfies the condition
\begin{itemize}
	\item[(V)] $V(x)\leq \frac{N(4-N)}{4} |x|^{-2}$ for $|x|\geq \rho$.
\end{itemize}

\begin{Lem}\label{lemLu}
	Let $u\in H^1(\R^N)$ and $N\in\{3,4\}$.
	Suppose that $A u\geq 0$, $u\geq 0$ and $(V)$ hold for $x\in \R^N\setminus B(0,\rho)$. Then 
	$\inf\{u(x): |x|=M\}=0$ for any $M\geq \rho$. If, in addition, $u\in \cC^1(\R^N\setminus \overline{B(0,\rho)})$, then $u=0$ .
\end{Lem}

Note that if $N=4$, then $V^+ = 0$ in $\R^N\setminus B(0,\rho)$.

\begin{proof}
	Suppose that $c:=\inf\{u(x): |x|=M\}>0$ for some $M\geq \rho$. We may assume $M=\rho$.
	Let $v(x):=c\rho^{\alpha}|x|^{-\alpha}$, and note that
	$\int_{|x|\geq\rho}|v|^2\,dx =\infty$, $\lim_{R\to\infty}\int_{R\leq |x|\leq R+1}|v|^2\,dx =0$ and $\lim_{R\to\infty}\int_{R\leq |x|\leq R+1}|\nabla v|^2\,dx =0$ provided $\frac{N-1}{2}<\alpha\leq \frac{N}{2}$. Choosing $\alpha = \frac N2$, we have
	\begin{eqnarray*}
		(-\Delta +V^+) v &=& c\rho^{\alpha}\big(-\alpha(\alpha+2-N)|x|^{-\alpha-2}+V^+(x)|x|^{-\alpha}\big)\\
		&\leq&
		c\rho^{\alpha}\Big(-\alpha(\alpha+2-N) +  \frac{N(4-N)}{4} \Big)|x|^{-\alpha-2}=0.
	\end{eqnarray*}
Therefore $(-\Delta +V^+)(u-v)\geq A u-(-\Delta +V^+)v\geq 0$ in $\R^N\setminus B(0,\rho)$ and $u-v\geq 0$ for $|x|=\rho$. Let $R>\rho$ and  $\eta_R\in \cC_0^{\infty}(\R^N)$ be a nonnegative function such that $|\nabla \eta_R|\leq C$, $\eta_R(x)=1$ for $x\in B(0,R)$ and $\eta_R(x)=0$ for $x\in \R^N\setminus B(0,R+1)$.  Testing the inequality $(-\Delta+V^+)(u-v)\ge 0$ with $(u-v)^-\eta_R\le 0$, we obtain 
	$$(-\Delta +V^+)(u-v)(u-v)^-\eta_R\leq 0$$
	 for $x\in \R^N\setminus B(0,\rho)$. Hence
	\begin{align}
	\label{ineq}	&  \int_{\R^N\setminus B(0,\rho)} \big(|\nabla (u-v)^-|^2+V^+(x)|(u-v)^-|^2\big)\eta_R\,dx\leq -\int_{\R^N\setminus B(0,\rho)} (\nabla (u-v)^- \nabla\eta_R)(u-v)^-\, dx\\
		\nonumber & \qquad\leq \frac12\int_{\R^N\setminus B(0,\rho)} \big(|\nabla (u-v)^-|^2+ |(u-v)^-|^2\big)|\nabla\eta_R|\, dx.
	\end{align}
 Since $\lim_{R\to\infty}\int_{R\leq |x|\leq R+1}|v|^2\,dx=\lim_{R\to\infty}\int_{R\leq |x|\leq R+1}|\nabla v|^2\,dx =0$, we obtain
	\begin{align}
		\label{ineq2} & \int_{\R^N\setminus B(0,\rho)} \big(|\nabla (u-v)^-|^2+ |(u-v)^-|^2\big)|\nabla\eta_R|\, dx \\
		\nonumber & \qquad \leq C\int_{B(0,R+1)\setminus B(0,R)} |\nabla (u-v)|^2+ |u-v|^2\, dx
		\to 0 \text{ as } R\to\infty.
	\end{align}
	It follows from \eqref{ineq}, \eqref{ineq2} and the Lebesgue dominated convergence theorem that
	$$\int_{\R^N\setminus B(0,\rho)}|\nabla (u-v)^-|^2\,dx=0,$$
	so $c_1:=(u-v)^-$ is constant for $|x|>\rho$. Note that
	$u-v=(u-v)^++c_1\geq c_1$, thus $u\geq v+c_1$ and
	$$\int_{|x|\geq \rho}|u|^2\,dx\geq \int_{|x|\geq \rho}|v+c_1|^2\,dx =\infty.$$
We get a contradiction since $u\in L^2(\R^N)$. Therefore	$\inf\{u(x): |x|=M\}=0$ for any $M\geq \rho$, hence $u\geq 0$. If $u\in \cC^1(\R^N\setminus \overline{B(0,\rho)})$, then $u=0$ on $\R^N\setminus \overline{B(0,\rho)}$ by the maximum principle.
\end{proof}

{\bf Acknowledgements.}
The authors would like to thank Louis Jeanjean for helpful comments and the referee of an earlier version of this paper for pointing out a gap in a proof and suggesting an improvement of Theorem \ref{th:PS}.

J. Mederski was partly supported by the National Science Centre, Poland (Grant No. 2020/37/B/ST1/02742).

{\bf Conflict of interest.}
 On behalf of all authors, the corresponding author states that there is no conflict of interest.
 
 {\bf Data availability statement.} No data were generated or analysed as part of this manuscript.


\begin{thebibliography}{99}
\baselineskip 2 mm

\bibitem{BahriLions} A. Bahri, P.L. Lions: \emph{Solutions of Superlinear Elliptic Equations and their Morse Indices}, Comm. Pure Appl. Math. \textbf{45} (1992), 1205--1215.

\bibitem{BartschJeanjean} T. Bartsch, L. Jeanjean: \emph{Normalized solutions for nonlinear Schr\"odinger systems}, Proc. Royal Soc. Edinburgh \textbf{148A} (2018), 225--242.

\bibitem{BartschJS2016}T. Bartsch, L. Jeanjean, N. Soave:
{\em Normalized solutions for a system of coupled cubic Schr\"odinger equations on $\R^3$}, J. Math. Pures Appl., {\bf 106} (4) (2016), 583--614.

\bibitem{BartschSoaveJFA} T. Bartsch, N. Soave: {\em A natural constraint approach to normalized solutions of nonlinear Schr\"odinger equations and systems}, J. Funct. Anal. {\bf 272} (12) (2017), 4998--5037. 

\bibitem{BartschSoaveJFACorr} T. Bartsch, N. Soave: Corrigendum: 
{\em Correction to: A natural constraint approach to normalized solutions of nonlinear Schr\"odinger equations and systems}, J. Funct. Anal., {\bf 275} (2) (2018), 516--521.

\bibitem{BartschSoaveCalV} T. Bartsch, N. Soave: {\em Multiple normalized solutions for a competing system of Schr\"odinger equations}, Calc. Var. {\bf 58} (2019), 22. 

\bibitem{BartschZhongZou} T. Bartsch, X. Zhong, W. Zou: {\em Normalized solutions for a coupled Schr\"odinger system}, Math. Ann. {\bf 380} (2021), 1713--1740.

\bibitem{benfort} V. Benci, D. Fortunato: \emph{A remark on the nodal regions of the solutions of some superlinear elliptic equations}, Proc. Royal Soc. Edinburgh \textbf{111A} (1989), 123--128. 


\bibitem{BerLions} H. Berestycki, P.-L. Lions: {\em 
Nonlinear scalar field equations. I. Existence of a ground state},
Arch. Rational Mech. Anal. {\bf 82} (1983), no. 4, 313--345.

\bibitem{BiegMed} B. Bieganowski, J. Mederski: {\em Normalized ground states of the nonlinear Schr\"odinger equation with at least mass critical growth},  J. Funct. Anal. {\bf 280} (2021), no. 11, 108989.
	
\bibitem{BrezisLiebCMP84} H. Brezis, E. Lieb: {\em Minimum action solutions of some vector field equations}, Comm. Math. Phys. {\bf 96} (1984), no. 1, 97--113.

\bibitem{Buryak} A.V. Buryak, P.D. Trapani, D.V. Skryabin, S. Trillo: {\em Optical solitons due to
	quadratic nonlinearities: from basic physics to futuristic applications}, Physics Reports {\bf 370} (2002), no. 2, 63--235.
		
\bibitem{Cwikel} M. Cwikel: {\em Weak type estimates for singular values and the number of bound states of Schr\"odinger operators}, Ann. Math. {\bf 106} (1977), 93--102.

\bibitem{DancerWeiWeth} E.N. Dancer, J. Wei, T. Weth: {\em A priori bounds versus multiple existence of positive solutions for a nonlinear Schr\"odinger system}, Ann. lnst. H. Poincar\'e, Analyse Non Lin\'eaire {\bf 27}(3) (2010), 953--969.

\bibitem{Esry} B.D. Esry, C.H. Greene, J.P. Burke, Jr.,  J.L. Bohn: {\em Hartree-Fock theory for double condensates}, Phys. Rev. Lett. {\bf 78} (19) (1997), 3594--3597.

\bibitem{flw} R.L. Frank, A. Laptev, T. Weidl: \emph{Schr\"odinger Operators: Eigenvalues and Lieb-Thirring Inequalities}, Cambridge University Press, Cambridge, 2023.

\bibitem{Frantzeskakis} D.J. Frantzeskakis: {\em Dark solitons in atomic Bose-Einstein condensates: from theory to experiments.}, J. Phys. A: Math. Theor. {\bf 43} (2010).
 
\bibitem{GarofaloLin} N. Garofalo, F.-H. Lin: {\em Unique continuation for elliptic operators: a geometric-variational approach}, Comm. Pure Appl. Math. {\bf 40} (1987), no. 3, 347--366.

\bibitem{GouJeanjean} T.  Gou, L. Jeanjean: {\em Multiple positive normalized solutions for nonlinear Schr\"odinger systems}, Nonlinearity {\bf 31} (2018), 2319.

\bibitem{Hardy} G.H. Hardy, J.E. Littlewood, G. P\'olya: \emph{Inequalities}, Cambridge University Press, London (1934).

\bibitem{Ikoma} N. Ikoma: {\em Compactness of minimizing sequences in nonlinear Schr\"odinger systems under multicostraint conditions}, Adv. Nonlinear Stud. {\bf 14} (1) (2014), 115--136.

\bibitem{IkomaTanaka} N. Ikoma, K. Tanaka: \emph{A note on deformation argument for $L^2$ normalized solutions of nonlinear Schr\"odinger equations and systems}, Adv. Diff. Eq. {\bf 24} (2019), 609--646.

\bibitem{Jeanjean} L. Jeanjean: {\em Existence of solutions with prescribed norm for semilinear elliptic equations}, Nonlinear Anal. \textbf{28} (10) (1997), 1633–1659.

\bibitem{JeanjeanLuNorm}  L. Jeanjean, S.-S. Lu: {\em A mass supercritical problem revisited}, Calc. Var. Partial Differential Equations {\bf 59}, 174, 43 pp. (2020).

\bibitem{Kwong} M.K. Kwong: {\em Uniqueness of positive solutions of $\Delta u-u+u^p=0$ in $\R^N$}, Arch. Rational Mech. Anal. {\bf 105} (1989), no. 3, 243--266.

\bibitem{LiZou} H. Li, W. Zou: {\em Normalized ground states for semilinear elliptic systems with critical and subcritical nonlinearities}, J. Fixed Point Theory Appl. 23, 43 (2021).

\bibitem{Lieb} E.H. Lieb: {\em The number of bound states of one-body Schr\"odinger operators and the Weyl problem}, Proc. Sym. Pure Math. {\bf 36} (1980), 241--252.

\bibitem{LSSY} E. H. Lieb, R. Seiringer, J. P. Solovej, J. Yngvason: {\em The Mathematics of the Bose Gas and its Condensation}, Birk\"auser, Basel, 2005.
		
\bibitem{Lions84} P.-L. Lions: {\em The concentration-compactness principle in the calculus of variations. The locally compact case. Part II}, Ann. Inst. H. Poincaré, Anal. Non Linéaire {\bf 1} (1984), 223--283.

\bibitem{LinWei} T.C. Lin, J. Wei:  \emph{Ground state of $N$ coupled nonlinear Schr\"odinger equations in $\R^n$, $n \le 3$}, Comm. Math. Phys. {\bf 255} (2005), 629--653.

\bibitem{Malomed} B. Malomed: {\em Multi-component Bose-Einstein condensates: Theory}. In: P.G. Kevrekidis, D.J. Frantzes\-kakis, R. Carretero-Gonzalez (Eds.): {\em Emergent Nonlinear Phenomena in Bose-Einstein Condensation}, Springer-Verlag, Berlin, 2008, 287--305.

\bibitem{MederskiSchino} J. Mederski, J. Schino: {\em Least energy solutions to a cooperative system of Schrödinger equations with prescribed $L^2$-bounds: at least $L^2$-critical growth}, Calc. Var. Partial Differential Equations {\bf 61}:10 (2022).

\bibitem{Rozenblum} G.V. Rozenblum: {\em Distribution of the discrete spectrum of singular differential operators}, Soviet Math. Dokl. {\bf 13} (1972), 245--249, and Soviet Math. {\bf  20} (1976), 63--71.

\bibitem{Schino} J. Schino: {\em Normalized ground states to a cooperative system of Schr\"odinger equations with generic $L^2$-subcritical or $L^2$-critical nonlinearity}, Adv. Differential Equations {\bf 27} (2022), no.7-8, 467--496.

\bibitem{Sirakov} B. Sirakov: \emph{Least energy solitary waves for a system of nonlinear Schr\"odinger equations in $\R^n$}, Comm. Math. Phys. \textbf{271} (2007), 199--221.

\bibitem{Strauss} W.A. Strauss: {\em Existence of solitary waves in higher dimensions}, Commun. Math. Phys. {\bf 55} (1977), 149--162.

\bibitem{Struwe} M. Struwe: {\em Variational Methods}, Springer, Berlin, 2008.

\bibitem{Stuart} C.A. Stuart: {\em Bifurcation for Dirichlet problems without eigenvalues}, Proc. Lond. Math. Soc. {\bf 45} (1982), 169--192.

\bibitem{Weinstein} M.I. Weinstein: \emph{Nonlinear Schr\"odinger equations and sharp interpolation estimates}, Comm. Math. Phys. \textbf{87} (1983), 567--576. 

\bibitem{Willem} M. Willem: {\em Minimax Theorems},  Birkh\"auser (1996).

\end{thebibliography}
\end{document}